\documentclass[a4paper,11pt]{amsart}

\usepackage[utf8]{inputenc}
\usepackage{a4wide}
\usepackage{amsmath,amssymb,amsfonts,amsthm,mathtools}
\usepackage{booktabs}
\usepackage{graphicx}
\usepackage{hyperref}
\usepackage{placeins}
\usepackage{subcaption}

\DeclareMathOperator{\rank}{rank}
\newcommand{\F}{\mathbb{F}}

\newcommand{\ba}{\boldsymbol{a}}
\newcommand{\bh}{\boldsymbol{h}}
\newcommand{\bk}{\boldsymbol{k}}

\newcommand{\bx}{\boldsymbol{x}}
\newcommand{\cD}{\mathcal{D}}
\newcommand{\cC}{\mathcal{C}}
\newcommand{\cP}{\mathcal{P}}
\newcommand{\wal}{\operatorname{wal}}

\newtheorem{definition}{Definition}[section]
\newtheorem{proposition}[definition]{Proposition}
\newtheorem{theorem}[definition]{Theorem}

\title[Digital Nets on Cubature Nodes]{Digital Nets on Cubature Nodes:\\
Inheriting Cubature Accuracy on Low-Dimensional Projections}

\author{Takehito Yoshiki}

\begin{document}
\sloppy

\begin{abstract}
Base-2 digital nets are practical high-dimensional integration rules: the sample
budget $N=2^m$ can be chosen independently of the ambient dimension, and the
generating matrices provide algebraic control of projections and Walsh-dual
weights. They are therefore well suited to problems whose error is governed by
weighted or low-dimensional projection structure. However, when one restricts
attention to a smooth low-dimensional projected component, a low-dimensional
cubature rule with a comparable number of nodes can be substantially more
accurate than the projected digital-net points. This raises the question of
whether low-dimensional cubature accuracy can be inserted into a
high-dimensional digital-net rule without forming the full tensor product.

We answer this question by a simple coordinate embedding: read the leading
$p$ binary digits of each coordinate as an index into $2^p$ equal-weight
cubature nodes, and replace the coordinate by the indexed node. When a
projection forms the full $p$-bit grid, the transformed rule coincides on that
projection with the corresponding product cubature rule; small projected
$t$-values provide sufficient conditions for such full-grid recovery. For
general integrands, the error separates into the corresponding product
cubature error and a residual digital-net term. Experiments with scrambled
Sobol' nets in dimension $50$ illustrate this mechanism and show finite-budget
improvements for the smooth low-order and coordinate-decaying test functions
considered here.
\end{abstract}


\maketitle

\section{Introduction}\label{sec:introduction}

Base-2 digital nets are a standard tool for high-dimensional quasi-Monte Carlo
(QMC) integration. A digital net places $N=2^m$ points in dimension $d$
without tensor-product growth, and its generating matrices provide algebraic
control of low-dimensional projections. This control is commonly expressed
through $t$-values, equidistribution properties, and Walsh-dual weights
\cite{Dick2008,DickKuoSloan2013,JoeKuo2008}. For high-dimensional problems
whose error is governed by weighted or low-dimensional projection structure,
this makes digital nets both practical and analyzable.

At the same time, if one isolates a smooth low-dimensional projected component
as a low-dimensional integration problem, a cubature rule with a comparable
number of nodes can be substantially more accurate than the projected
digital-net points. The aim of this paper is to keep the digital net as the
high-dimensional rule and add cubature accuracy to its low-dimensional
projections, without forming the full tensor product of cubature nodes.

The construction is simple. Fix a resolution $p$ and an equal-weight
one-dimensional quadrature rule with $2^p$ nodes, whose accuracy is
characterized by polynomial exactness,
\[
  \cC_p=\{c_0,\ldots,c_{2^p-1}\}\subset[0,1].
\]
For each coordinate of a digital-net point, read the leading $p$ binary digits
as an index $a\in\{0,\ldots,2^p-1\}$, and replace that coordinate by $c_a$.
The digital net still supplies the high-dimensional finite-field pattern; the
cubature rule supplies the coordinate alphabet.

The key mechanism is full-grid recovery on projections. If, on a coordinate
subset, the leading $p$-bit indices form the full product grid, then the
transformed rule agrees on that subset with the corresponding product cubature
rule. Thus every function depending only on that subset is evaluated exactly
as by product cubature. Projected $t$-values give simple sufficient conditions
for this recovery. For general integrands, the error separates into the
corresponding product cubature error and a residual digital-net term.

Viewed abstractly, the construction combines a high-dimensional finite pattern
with cubature nodes in each coordinate. This form has precedents in cubature:
orthogonal-array and coding-theoretic constructions also place finite-field or
finite-alphabet patterns on product cubature nodes \cite{Kuperberg2006}. The
finite-pattern structure at fixed $p$ is therefore close to those constructions.
The contribution here is not a new orthogonal-array cubature construction.
Those constructions start from cubature and use the finite pattern to build a
high-dimensional formula, typically with polynomial exactness guarantees. Here
we take an existing digital net as input and use its projection ranks,
$t$-values, and Walsh-dual structure to identify when a cubature alphabet
recovers particular low-dimensional projections and how the remaining residual
is organized.

The main contributions are:
\begin{itemize}
  \item a plug-in cubature-node embedding for existing base-2 digital nets;
  \item an exact error decomposition into product cubature error and a
  residual digital-net term;
  \item a Walsh description of the residual using the ordinary infinite dual
  of the underlying digital net and a retained/tail digit split analogous to
  that used for polynomial lattice rules \cite{DickGodaYoshiki2019};
  \item a full-grid recovery criterion, with projected $t$-values as
  sufficient conditions, showing when low-dimensional projections coincide
  with the corresponding product cubature rules;
  \item numerical experiments with scrambled Sobol' nets in dimension $50$ for
  smooth low-order and coordinate-decaying test functions.
\end{itemize}

\section{Cubature-node embedding}\label{sec:construction}

\subsection{The transformed rule}

Let
\[
  \cP_N=\{\bx_0,\ldots,\bx_{N-1}\}\subset[0,1)^d,
  \qquad N=2^m,
\]
be a digital net in base $2$. We follow the standard generating-matrix
notation for digital nets \cite{JoeKuo2008}.
We use the standard infinite-row convention: the matrices have $m$ input
columns and infinitely many output rows, so that each coordinate has an
infinite binary expansion. Thus $\cP_N$ is still a finite digital net of
$N=2^m$ points, rather than a digital sequence. Let
\[
  C_1,\ldots,C_d\in\F_2^{\mathbb N\times m}
\]
be its generating matrices. For
$n=n_0+2n_1+\cdots+2^{m-1}n_{m-1}$, write
$\vec n=(n_0,\ldots,n_{m-1})^\top$. The $j$th coordinate of the $n$th
point is generated by
\[
  \vec y_{n,j}=C_j\vec n\in\F_2^{\mathbb N},
  \qquad
  x_{n,j}=\sum_{r\ge1} y_{n,j,r}2^{-r}.
\]
We now introduce the coordinate replacement used in this paper. A
one-dimensional cubature rule is used coordinate-wise: the leading $p$ binary
digits of each coordinate select one cubature node.

Fix $p\ge1$ and an indexed equal-weight one-dimensional rule
\[
  Q_p^{(1)}(g)
  =
  2^{-p}\sum_{a=0}^{2^p-1}g(c_a).
\]
The index $a$ is identified with a $p$-bit binary symbol. For $x\in[0,1)$,
define
\[
  I_p(x)=\lfloor2^p x\rfloor,
  \qquad
  T_p(x)=c_{I_p(x)}.
\]
The map $T_p$ is applied coordinate-wise to points of the digital net. For
$\bx=(x_1,\ldots,x_d)$, set
\[
  T_p(\bx)=\bigl(T_p(x_1),\ldots,T_p(x_d)\bigr).
\]
We also write
\[
  \tau_p(x)=2^{-p}I_p(x)
\]
for the usual leading $p$-bit truncation.

\begin{definition}[Digital net on cubature nodes]
The transformed rule is
\[
  Q_{N,p}(f)
  =
  \frac1N\sum_{n=0}^{N-1}f(T_p(\bx_n)).
\]
\end{definition}

Repeated transformed points are allowed. Computationally, the rule only
requires extracting the first $p$ bits of each coordinate and applying the
corresponding one-dimensional cubature-node lookup.

The complete product rule associated with the same one-dimensional nodes is
\[
  Q_p^{\mathrm{prod}}(f)
  =
  2^{-pd}
  \sum_{\ba\in\{0,\ldots,2^p-1\}^d}
  f(c_{a_1},\ldots,c_{a_d}).
\]
This product rule is generally too large to evaluate. The digital net selects
only $N$ indexed combinations from its product alphabet. The next section gives
the exact condition under which these selected combinations reproduce a
low-dimensional product cubature rule.

\section{Projection recovery}\label{sec:projection}

\subsection{Rank criterion}

The recovery statement is most naturally expressed for projected components.
Let $u\subseteq\{1,\ldots,d\}$ be a coordinate subset, and let $f_u$ denote an
integrand component depending only on the coordinates in $u$. This notation is
consistent with the usual ANOVA decomposition $f=\sum_u f_u$
\cite{WangFang2003}, although the
rank criterion below applies to any function depending only on $u$.

The transformation changes the coordinate values, but not the leading
$p$-bit symbol pattern of the digital net. This retained finite pattern is the
mechanism behind cubature accuracy on low-dimensional projections. If a
component $f_u$ depends only on coordinates in $u$ and the leading $p$-bit
symbols on $u$ are uniformly covered, then the transformed digital net does not
merely approximate the product cubature rule on that component; it coincides
with it:
\[
  Q_{N,p}(f_u)=Q_{p,u}^{\mathrm{prod}}(f_u).
\]
Here $Q_{p,u}^{\mathrm{prod}}$ denotes the $|u|$-dimensional tensor product of
the same one-dimensional cubature rule, and $I_u$ denotes integration over the
coordinates in $u$.
Consequently its error on that component is exactly the product-cubature bias
\[
  Q_{N,p}(f_u)-I_u(f_u)
  =
  Q_{p,u}^{\mathrm{prod}}(f_u)-I_u(f_u).
\]
The rank criterion below is the finite-field condition for this equality.

The finite-field object behind this condition is the map from the $m$ input
bits of the digital net to the first $p$ output bits. These digits define a
linear map
\[
  A_p:\F_2^m\longrightarrow(\F_2^p)^d,
\]
whose coordinate projection onto $u$ is denoted by
\[
  A_{p,u}:\F_2^m\longrightarrow(\F_2^p)^{|u|}
\].

\begin{proposition}[Exact rank criterion]\label{prop:rank}
The projected leading-digit pattern is uniform on
$(\F_2^p)^{|u|}$ if and only if
\[
  \rank(A_{p,u})=p|u|.
\]
When this condition holds, the transformed digital net agrees with the
complete product cubature rule on every function depending only on $u$; hence
\[
  Q_{N,p}(f_u)-I_u(f_u)
  =
  Q_{p,u}^{\mathrm{prod}}(f_u)-I_u(f_u).
\]
\end{proposition}

\begin{proof}
A linear map has uniform fibers over its image. It covers every projected
symbol vector exactly when it is surjective, which is equivalent to full rank.
Under this condition each symbol vector in
$(\F_2^p)^{|u|}$ occurs with frequency $2^{-p|u|}$. Grouping the transformed
points by these symbols therefore gives exactly the complete product cubature
average on $u$.
\end{proof}

For an ANOVA decomposition $f=\sum_u f_u$, every component whose projected
matrix has full rank inherits the corresponding product-cubature accuracy.
This is the main mechanism by which low-dimensional cubature accuracy can be
inherited inside a high-dimensional digital net.

\subsection{Projection-specific \texorpdfstring{$t$}{t}-values}

We use the standard definition of the $t$-value for the projected digital net
\cite{JoeKuo2008}. It provides a convenient sufficient certificate for
full-grid recovery. Let $t_u$ be a valid $t$-value of the projected digital net
on $u$. If $p|u|\leq m-t_u$, then every $p$-bit box in that projection has equal
occupancy, and hence the full-grid recovery condition holds.

\begin{proposition}[$t$-value sufficient condition]\label{prop:tvalue}
If
\[
  p|u|\leq m-t_u,
\]
then $\rank(A_{p,u})=p|u|$. Consequently, the transformed rule reproduces the
complete product cubature rule on $u$.
\end{proposition}

\begin{proof}
Every projected resolution-$p$ binary box can be partitioned into elementary
intervals of volume $2^{t_u-m}$. The $(t_u,m,|u|)$-net property gives equal
occupancy of these boxes and therefore equal occupancy of the coarser
resolution-$p$ boxes.
\end{proof}

The condition is sufficient but not necessary. The rank criterion can certify
exact recovery even when the available projected $t$-value bound does not. For
components that are not fully recovered, the remaining error is described by
the Walsh residual in the next section.

\section{Error structure}\label{sec:error}

\subsection{Product-cubature bias and residual digital-net error}

Choose an order-preserving extension
\[
  \phi_p:[0,1]\to[0,1],
  \qquad
  \phi_p(a/2^p)=c_a,\quad 0\leq a<2^p,
\]
and put
\[
  F_p(\bx)=f(\phi_p(x_1),\ldots,\phi_p(x_d)).
\]
The interpolation condition gives $T_p=\phi_p\circ\tau_p$ coordinate-wise.
Thus, let
\[
  g_p(\bx)=f(T_p(\bx))=F_p(\tau_p(\bx)).
\]
The identities in this subsection are algebraic and require only the values of
$f$ at the cubature nodes. For the Walsh representations below, we assume that
the Walsh expansion of the chosen extension $F_p$ may be summed termwise at the
relevant grid points. On the resolution-$p$ dyadic box indexed by
$\ba=(a_1,\ldots,a_d)$, the function $g_p$ is constant with value
$f(c_{a_1},\ldots,c_{a_d})$. Since each such box has volume $2^{-pd}$,
\[
  I(g_p)=Q_p^{\mathrm{prod}}(f).
\]
Moreover $Q_{N,p}(f)=Q_{\cP_N}(g_p)$. Therefore
\begin{equation}
  \boxed{
  Q_{N,p}(f)-I(f)
  =
  \underbrace{Q_p^{\mathrm{prod}}(f)-I(f)}_{\text{product-cubature bias}}
  +
  \underbrace{Q_{\cP_N}(g_p)-I(g_p)}_{\text{digital-net residual for }g_p}.
  }
  \label{eq:main-decomposition}
\end{equation}
This identity is the main error decomposition. The first term is a product
cubature error for the original integrand $f$. The second term is not the
ordinary digital-net error for $f$; it is the digital-net error for the
composed function $g_p=f\circ T_p$. The purpose of the remainder of this
section is to compare this residual with the QMC error
$Q_{\cP_N}(f)-I(f)$ of the unmodified digital net. We do so at the Walsh level:
the latter is a dual sum of the Walsh coefficients of $f$, while the former
will become a sum of the coefficients of $F_p$ subject to a retained-dual
condition determined by the same generating matrices.

\subsection{Digital-net dual notation}

The residual term in \eqref{eq:main-decomposition} is a digital-net quadrature
error for $g_p$. We now express it through the ordinary infinite Walsh dual of
the original net. This will provide the representation that is split into
retained and tail digits in the next subsection.

Use the same generating matrices $C_1,\ldots,C_d$ as in
Section~\ref{sec:construction}.
For $k_j\in\mathbb N_0$, let $\vec k_j$ be its binary digit vector, padded by
zeros. The full Walsh dual of the digital net is the infinite set
\cite{Dick2008}
\[
  \cP_N^\perp
  =
  \left\{
  \bk\in\mathbb N_0^d:
  C_1^\top\vec k_1+\cdots+C_d^\top\vec k_d=0
  \right\}.
\]
For any function $g$ whose Walsh expansion may be summed termwise, the
character property gives \cite{Dick2008}
\[
  Q_{\cP_N}(g)-I(g)
  =
  \sum_{\bk\in\cP_N^\perp\setminus\{\boldsymbol{0}\}}\widehat g(\bk).
\]
Applied directly to $g_p$, the same formula gives the residual term in
\eqref{eq:main-decomposition}. Its coefficients are those of a $p$-bit step
function, however, so this form does not yet make the comparison above
transparent.

\subsection{Retained and tail digits}

To compare the residual with the usual Walsh error mechanism of the original
digital net, rewrite the $p$-bit step function $g_p$ through the chosen
extension $F_p$. The actual estimator then becomes
\[
  Q_{N,p}(f)=Q_{\cP_N}(F_p\circ\tau_p),
  \qquad
  Q_p^{\mathrm{prod}}(f)=Q_{\mathrm{grid},p}(F_p),
\]
where $Q_{\mathrm{grid},p}$ is the full dyadic product grid of resolution $p$.

For $\bk\in\mathbb N_0^d$, split each component into retained and tail parts:
\[
  \bk=\bk^{(0)}+2^p\bk^{(1)},
  \qquad
  \bk^{(0)}\in\{0,\ldots,2^p-1\}^d.
\]
Let $C_{j,p}$ denote the first $p$ rows of $C_j$ and define the retained
dual-character set
\[
  \cD_{N,p}^{\perp}
  =
  \left\{
  \bh\in\{0,\ldots,2^p-1\}^d:
  C_{1,p}^\top\vec h_1+\cdots+C_{d,p}^\top\vec h_d=0
  \right\}.
\]
This is only a set of retained characters. The residual Walsh frequencies
below are still infinite because the tail $\bk^{(1)}$ is unrestricted.

\begin{theorem}[Retained/tail Walsh split]\label{thm:retained-tail}
Assume that the Walsh expansion of $F_p$ can be summed termwise at the
resolution-$p$ grid points. Then
\[
  Q_{N,p}(f)
  =
  \sum_{\substack{\bk\in\mathbb N_0^d\\
                  \bk^{(0)}\in\cD_{N,p}^{\perp}}}
  \widehat{F_p}(\bk),
  \qquad
  Q_p^{\mathrm{prod}}(f)
  =
  \sum_{\substack{\bk\in\mathbb N_0^d\\
                  \bk^{(0)}=\boldsymbol{0}}}
  \widehat{F_p}(\bk).
\]
Consequently,
\begin{equation}
  Q_{N,p}(f)-I(f)
  =
  \underbrace{Q_p^{\mathrm{prod}}(f)-I(f)}_{\text{product-cubature bias}}
  +
  \underbrace{
  \sum_{\substack{\bk\in\mathbb N_0^d\\
                  \bk^{(0)}\in
                  \cD_{N,p}^{\perp}\setminus\{\boldsymbol{0}\}}}
  \widehat{F_p}(\bk)
  }_{\text{residual retained-dual contribution}}.
  \label{eq:retained-tail-error}
\end{equation}
\end{theorem}

\begin{proof}
We separate the two averages.

First consider the digital-net average. Push the truncation onto the point
set: if $\tau_p(\cP_N)$ denotes the image multiset, then
\[
  Q_{N,p}(f)
  = Q_{\cP_N}(F_p\circ\tau_p)
  = Q_{\tau_p(\cP_N)}(F_p).
\]
Thus $F_p$, rather than the step function $g_p=F_p\circ\tau_p$, can be
expanded in Walsh characters. For any Walsh frequency $\bk$,
\[
  \wal_{\bk}(\tau_p(\bx))
  =
  \wal_{\bk^{(0)}}(\bx),
\]
because $\tau_p$ sets the input digits after position $p$ to zero. Thus, in
the binary Walsh pairing, the tail digits $\bk^{(1)}$ of the Walsh index
multiply only zero input digits and do not enter the sample average. Expanding
$F_p$ on the truncated point multiset and then using this identity gives
\[
  Q_{N,p}(f)
  = Q_{\tau_p(\cP_N)}(F_p)
  = \sum_{\bk\in\mathbb N_0^d}\widehat{F_p}(\bk)
    Q_{\tau_p(\cP_N)}(\wal_{\bk})
  =
  \sum_{\bk\in\mathbb N_0^d}\widehat{F_p}(\bk)
  Q_{\cP_N}(\wal_{\bk^{(0)}}).
\]
When the character property is applied to a retained character, its digit
vector has zero entries after position $p$; hence the full dual test reduces
to the first $p$ rows of the generating matrices. It gives, for every retained character
$\bh\in\{0,\ldots,2^p-1\}^d$,
\[
  Q_{\cP_N}(\wal_{\bh})
  =
  \begin{cases}
  1, & \bh\in\cD_{N,p}^{\perp},\\
  0, & \bh\notin\cD_{N,p}^{\perp}.
  \end{cases}
\]
Therefore only terms with $\bk^{(0)}\in\cD_{N,p}^{\perp}$ remain, giving the
first formula.

Second consider the full product grid. Since
$F_p(\ba/2^p)=f(c_{a_1},\ldots,c_{a_d})$, the product cubature value is
\[
  Q_p^{\mathrm{prod}}(f)=Q_{\mathrm{grid},p}(F_p).
\]
On this grid, the Walsh average of $\wal_{\bk}$ is one when
$\bk^{(0)}=\boldsymbol{0}$ and zero otherwise, i.e.
\[
  Q_{\mathrm{grid},p}(\wal_{\bk})
  =
  \begin{cases}
  1, & \bk^{(0)}=\boldsymbol{0},\\
  0, & \bk^{(0)}\ne\boldsymbol{0}.
  \end{cases}
\]
This gives the second formula.
Adding and subtracting this product-grid value yields
\eqref{eq:retained-tail-error}.
\end{proof}

Equation~\eqref{eq:retained-tail-error} is the analogue of the split in
\cite{DickGodaYoshiki2019}: the product term corresponds to frequencies whose
retained part is zero, while the residual consists of frequencies whose
retained part is a nonzero dual character. The tail digits are not restricted.
Thus the second term should not be interpreted as an error over a finite dual
set.

\subsection{Dual weights and coefficient distortion}

We now compare the retained-dual residual in
\eqref{eq:retained-tail-error} with the ordinary Walsh error of the unmodified
digital net. There is no direct term-by-term comparison between the two,
because both the summation range in $\bk$ and the coefficients attached to
those indices have changed. The ordinary digital-net error sums the Walsh
coefficients of $f$ over the full dual $\cP_N^\perp$, whereas the residual is
organized by retained indices and uses the Walsh coefficients of the composed
function $F_p$.

The summation range over $\bk$ has a favorable property: for the digit weights
$\mu_\alpha$ used in Walsh-coefficient bounds for smooth functions, its minimum
weight cannot be lower than the minimum dual weight of the original digital
net. Thus the cubature-node embedding does not introduce new lower-weight
Walsh indices. The possible adverse effect is on the coefficient side: the
residual uses the Walsh coefficients of $F_p$ rather than those of $f$.

We first make the summation-range statement precise.

We recall the digit weights only to fix notation; these are standard in the
digital-net literature \cite{Dick2008}. If
\[
  k_j=\sum_{\ell=1}^{v_j}2^{a_{j,\ell}-1},
  \qquad
  a_{j,1}>\cdots>a_{j,v_j}>0,
\]
then the Dick weight of order $\alpha$ is
\[
  \mu_\alpha(\bk)
  =
  \sum_{j=1}^d
  \sum_{\ell=1}^{\min(\alpha,v_j)} a_{j,\ell}.
\]
The NRT weight is the case $\alpha=1$, while
$\mu_\infty(\bk)=\sum_{j=1}^d\sum_{\ell=1}^{v_j}a_{j,\ell}$ counts all nonzero
digit positions with their positions. The only structural property used in
Proposition~\ref{prop:dual-weight} is monotonicity under adding tail digits.

\begin{proposition}[Inherited dual-weight exclusion]
\label{prop:dual-weight}
Let $\mu$ be a digit weight that is monotone under adding tail digits, such as
the NRT weight, $\mu_\infty$, or Dick weights $\mu_\alpha$. Define
\[
  \rho_\mu(\cP_N)
  =
  \min_{\bk\in\cP_N^\perp\setminus\{\boldsymbol{0}\}}\mu(\bk).
\]
Then every frequency in the residual set of
\eqref{eq:retained-tail-error} satisfies
\[
  \mu(\bk)\ge \rho_\mu(\cP_N).
\]
\end{proposition}

\begin{proof}
Let $\bk$ appear in the residual set and write
$\bk=\bk^{(0)}+2^p\bk^{(1)}$. Then
$\bk^{(0)}\in\cD_{N,p}^{\perp}\setminus\{\boldsymbol{0}\}$. The frequency
$\bk^{(0)}$, with zero tail digits, belongs to the full dual
$\cP_N^\perp$, because the full dual condition uses only the first $p$ rows
when the higher digits are zero. Hence
$\mu(\bk^{(0)})\ge\rho_\mu(\cP_N)$. By monotonicity under adding tail digits,
$\mu(\bk)\ge\mu(\bk^{(0)})$.
\end{proof}

We now return to the coefficient side. The residual coefficients are the Walsh
coefficients of the fixed composed function $F_p$, not the original Walsh
coefficients of $f$. This is the main possible source of deterioration relative
to the unmodified digital net.

We do not attempt to give a sharp general bound predicting this coefficient
effect. Such a bound would depend on the cubature nodes, their labeling, the
chosen smooth extension $\phi_p$, the integrand, and the particular retained
indices that survive. The relevant design issue is that the cubature alphabet
should admit an extension that does not create unnecessary oscillation on the
retained symbol grid. In particular, it is preferable for the map
$a/2^p\mapsto c_a$ to be order-preserving and not too far from the identity,
and for the node symmetries to be compatible with low-complexity Walsh
characters. Thus the construction is not a universal improvement theorem over
the original digital net. It trades a summation range with controlled minimum
weight against the coefficient effect induced by the cubature alphabet.

\section{Numerical experiments}\label{sec:experiments}

\subsection{Setup}

We use dimension $d=50$ and scrambled Sobol' points generated by
\texttt{scipy.stats.qmc.Sobol} with $52$ output bits, using the
Joe--Kuo direction numbers \cite{JoeKuo2008} and LMS+shift scrambling
\cite{Matousek1998}. The $52$ output bits are chosen so that the resulting dyadic
coordinates are exactly representable in IEEE~754 binary64 arithmetic, whose
significand has $53$ bits. The sample sizes are
$N=2^m$, $m=8,\ldots,17$, with $32$ independent scramblings. We compare the
original full-precision scrambled Sobol' estimator with the proposed
coordinate embedding at resolutions $p=3$ and $p=4$, using respectively
$8$-node and $16$-node equal-weight one-dimensional rules.
The cubature-node vectors and their lookup order are given in
Appendix~\ref{app:node-vectors} and fixed in the accompanying experiment script.
For replicate $r$ at budget exponent $m$, it generates the Sobol' points with
\texttt{qmc.Sobol(d=50, scramble=True, bits=52, seed=20260611+10000r+m)}.
For the nested LMS+shift scrambling used here \cite{Matousek1998}, the leading-digit
ranks, projected occupancies, and hence the full-grid recovery events are
identical across replicates. On the leading $p$ digits LMS acts coordinate-wise
by invertible lower-triangular binary matrices. Across any projection, these
form an invertible block-diagonal transformation of $\F_2^{p|u|}$; hence
$\rank(A_{p,u})$ is unchanged, while a digital shift merely translates the
$p$-bit box labels. Only the sampled function values and residual errors change.

For each function, method, and budget, we report the median absolute error
over the $32$ scramblings and its 25th--75th percentile range. This is a
finite-budget typical-error summary rather than an RMSE comparison: at a fixed
budget, the absolute-error distribution across scramblings can be skewed, and
a few unusually large errors can dominate a mean or RMSE. The summary therefore
describes a typical replicate and its central variability rather than a
mean-square risk. In particular, the unmodified scrambled Sobol' estimator is unbiased, whereas the
embedded rule has the deterministic product-cubature bias in
\eqref{eq:main-decomposition}.

For the test functions, set $g(x)=\exp(x)-(e-1)$ and consider the pure
interaction functions
\begin{align*}
  h_2(\bx)
  &=
  \binom{d}{2}^{-1}\sum_{i<j}g(x_i)g(x_j),\\
  h_3(\bx)
  &=
  \binom{d}{3}^{-1}\sum_{i<j<k}g(x_i)g(x_j)g(x_k),
\end{align*}
which isolate projection recovery at fixed interaction orders; their subscripts
indicate the interaction order. To test coordinate-decaying structure, we also
use the weighted product family
\[
  f_\alpha(\bx)
  =
  \exp\left(\sum_{j=1}^d j^{-\alpha}x_j\right),
  \qquad \alpha\in\{1,2,3\}.
\]
Increasing $\alpha$ concentrates the weighted product functions on
the leading coordinates; this is the same type of coordinate-importance
structure modeled by product weights in weighted QMC spaces
\cite{DickKuoSloan2013}.

The experiments are designed to test three predictions. Pure pair interactions
should show an abrupt transition once all pair projections recover the product
cubature rule. Pure triple interactions should be harder, because the required
rank conditions involve $3p$ retained bits and many more projections. Weighted
exponentials should benefit most when the coordinate weights decay quickly,
because then leading low-dimensional projections dominate the error.

The one-dimensional rules used to define the coordinate alphabet are symmetric
about $1/2$ and were obtained by solving moment-matching equations numerically
for equal weights. The $p=3$ rule is an $8$-node equal-weight rule, not the
$8$-point Gauss rule; with the node values used here it is exact for monomials
through degree $7$. The $p=4$ rule is exact through degree $11$, to numerical
precision. These exactness facts are only meant to describe the coordinate
alphabet used in the experiments, not to make a separate claim about optimal
one-dimensional quadrature.

\subsection{Projection-recovery transitions}

Figure~\ref{fig:pure} compares the pure second- and third-order functions
$h_2$ and $h_3$. For the pair function $h_2$, the transformed-rule error falls
abruptly to the product-rule floor once all pair projections have full rank. The transition
occurs at different budgets for $p=3$ and $p=4$, reflecting the cost of
covering $2p$ retained bits per pair. The third-order function depends on many
more projected constraints and does not exhibit an equally sharp transition;
the observed changes are correspondingly more modest. Thus the third-order plot
is not simply a failed version of the pair experiment; it is testing the point
at which the available Sobol' projections no longer recover all relevant
product-cubature patterns.

\begin{figure}[htbp]
  \centering
  \begin{subfigure}{0.49\linewidth}
    \includegraphics[width=\linewidth]{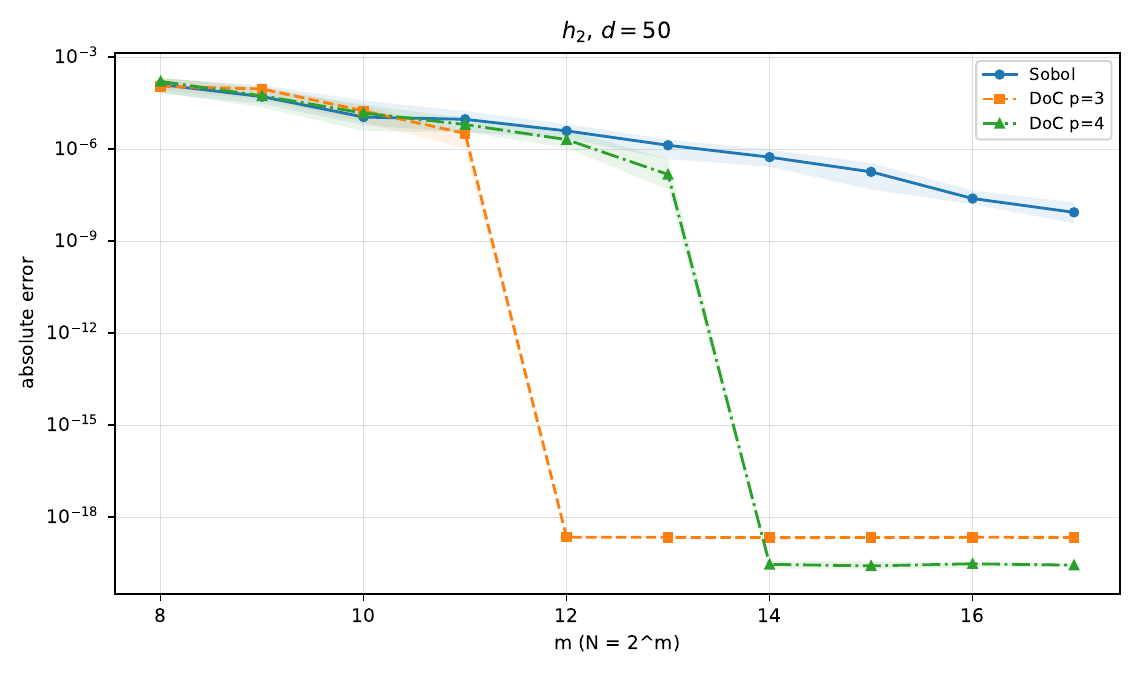}
    \caption{Pure second-order interaction $h_2$.}
  \end{subfigure}
  \begin{subfigure}{0.49\linewidth}
    \includegraphics[width=\linewidth]{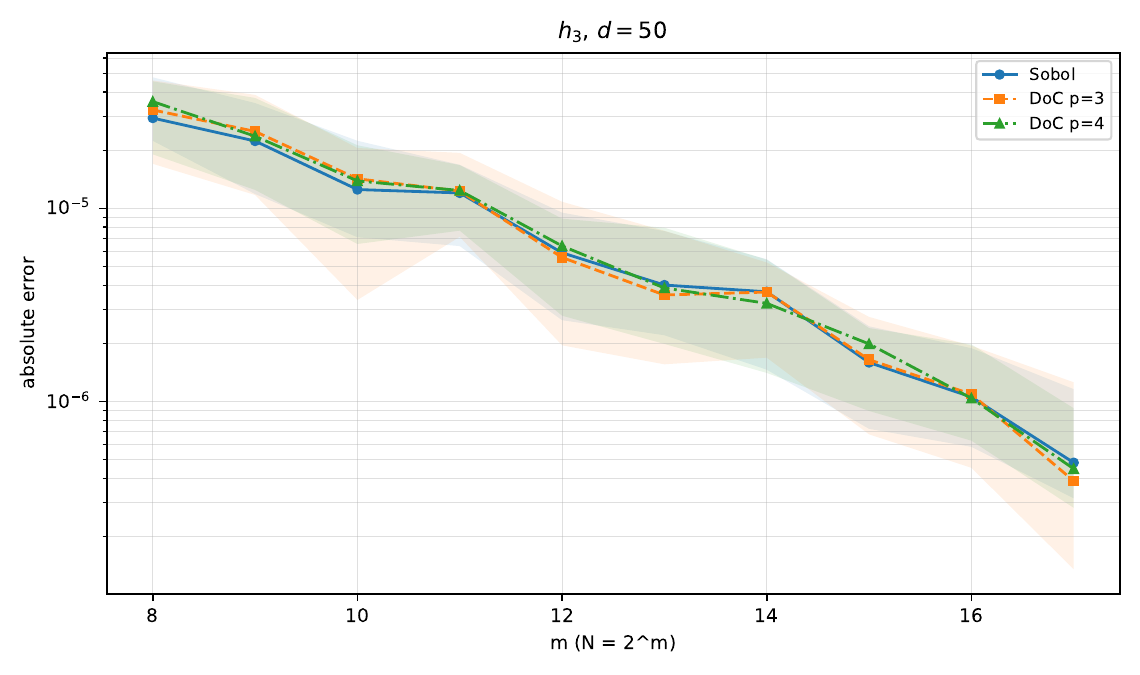}
    \caption{Pure third-order interaction $h_3$.}
  \end{subfigure}
  \caption{Absolute errors over $32$ scramblings. Lines show the median of the
  absolute errors, and shaded bands show their 25th--75th percentile range.}
  \label{fig:pure}
\end{figure}

To make the pair-transition mechanism explicit, Table~\ref{tab:occupancy}
counts the two-coordinate projections satisfying the full-grid recovery
condition. In dimension $d=50$ there are $\binom{50}{2}=1225$ such pair
projections.
The counts are invariant over the $32$ scramblings used in the experiment,
because LMS+shift scrambling preserves the leading-digit ranks and occupancies.
The observed error transitions agree with the complete pair-recovery budgets:
as long as even a small number of pair projections are unrecovered, the leading
pair terms can still dominate the error, while full recovery collapses that
contribution to the product-cubature floor.
The pair test is deliberately diagnostic: because $g$ is centered, $h_2$ is a
pure average of two-coordinate components. Its very small recovered errors
therefore verify the recovery mechanism and should not be interpreted as a
generic performance level for arbitrary integrands.

\begin{table}[htbp]
  \centering
  \caption{Pair projections with exact product-node occupancy.}
  \label{tab:occupancy}
  \begin{tabular}{rrrr}
    \toprule
    $m$ & $N$ & $p=3$ & $p=4$\\
    \midrule
     8 &    256 & 1030/1225 (84.1\%) &  427/1225 (34.9\%)\\
    10 &   1024 & 1205/1225 (98.4\%) & 1051/1225 (85.8\%)\\
    12 &   4096 & 1225/1225 (100.0\%) & 1210/1225 (98.8\%)\\
    14 &  16384 & 1225/1225 (100.0\%) & 1225/1225 (100.0\%)\\
    16 &  65536 & 1225/1225 (100.0\%) & 1225/1225 (100.0\%)\\
    \bottomrule
  \end{tabular}
\end{table}

\FloatBarrier

\subsection{Coordinate decay}

Figure~\ref{fig:weighted} shows the weighted exponential results for
$f_\alpha$, $\alpha=1,2,3$. The
coordinate embedding is most effective when important projections are
concentrated among the leading coordinates, matching the product-weight
intuition that later coordinates should contribute less to the error
\cite{DickKuoSloan2013}. At $m=17$, the best transformed
rule improves over full-precision scrambled Sobol' by factors of approximately
$1.15$, $5.61$, and $65.2$ for $\alpha=1,2,3$, respectively, as summarized
in Table~\ref{tab:improvement}.
For larger $\alpha$, the tail coordinates carry less weight, so exact or nearly
exact recovery of leading low-dimensional projections has a larger visible
effect.

\begin{table}[htbp]
  \centering
  \caption{Median absolute errors for the non-diagnostic test functions:
  best transformed rule versus full-precision Sobol' at $m=17$. Pure-pair
  $h_2$ values are omitted because their post-recovery plateaus are affected by
  floating-point cancellation.}
  \label{tab:improvement}
  \begin{tabular}{lrrr}
    \toprule
    Function & Sobol' & Best transformed rule & Improvement\\
    \midrule
    Pure triple $h_3$
      & $4.83{\times}10^{-7}$ & $3.89{\times}10^{-7}$ & $1.24$\\
    Weighted exp. $f_1$
      & $4.42{\times}10^{-6}$ & $3.86{\times}10^{-6}$ & $1.15$\\
    Weighted exp. $f_2$
      & $2.73{\times}10^{-9}$ & $4.86{\times}10^{-10}$ & $5.61$\\
    Weighted exp. $f_3$
      & $4.57{\times}10^{-11}$ & $7.01{\times}10^{-13}$ & $65.2$\\
    \bottomrule
  \end{tabular}
\end{table}

Figure~\ref{fig:weighted} also illustrates the trade-off in choosing $p$. A
smaller resolution requires fewer retained bits for full-grid recovery. Thus
$p=3$ can recover low-dimensional projections at smaller budgets, as
illustrated by the pure pair experiment in Figure~\ref{fig:pure}. The price is a larger fixed
product-cubature bias: once the residual part has been reduced, the error may
stop decreasing at the floor set by the $8$-node coordinate rule. A larger
resolution such as $p=4$ delays recovery, but gives a lower product-cubature
floor after the important projections have been recovered.

\begin{figure}[htbp]
  \centering
  \begin{subfigure}{0.48\linewidth}
    \includegraphics[width=\linewidth]{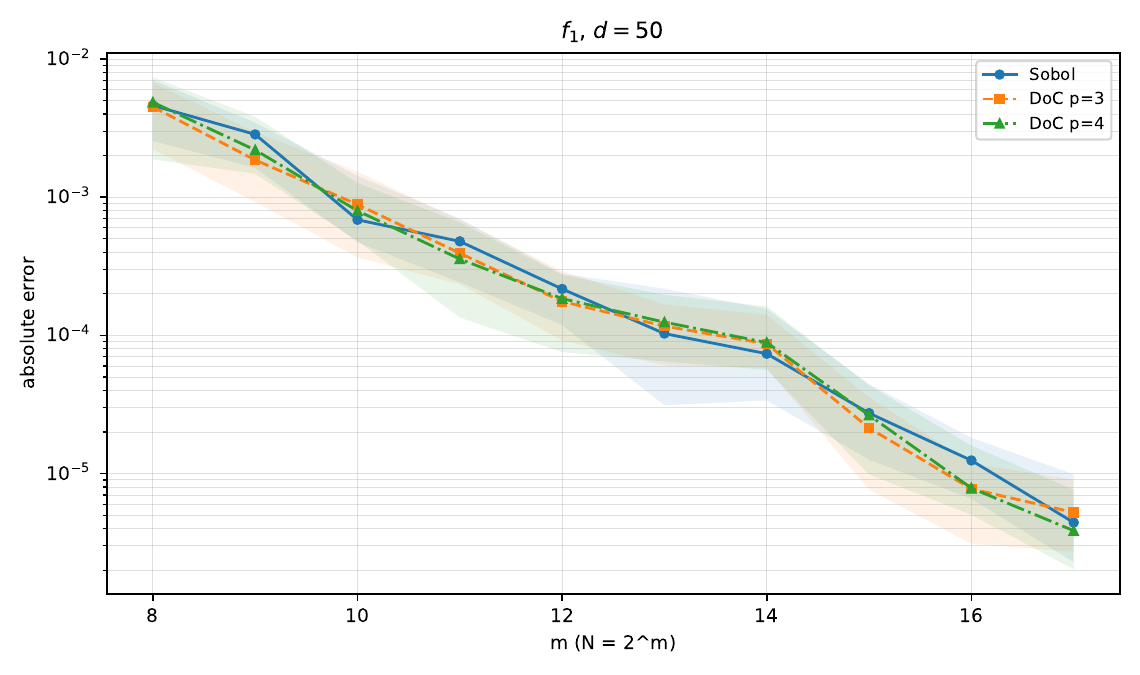}
    \caption{$\alpha=1$.}
  \end{subfigure}
  \hfill
  \begin{subfigure}{0.48\linewidth}
    \includegraphics[width=\linewidth]{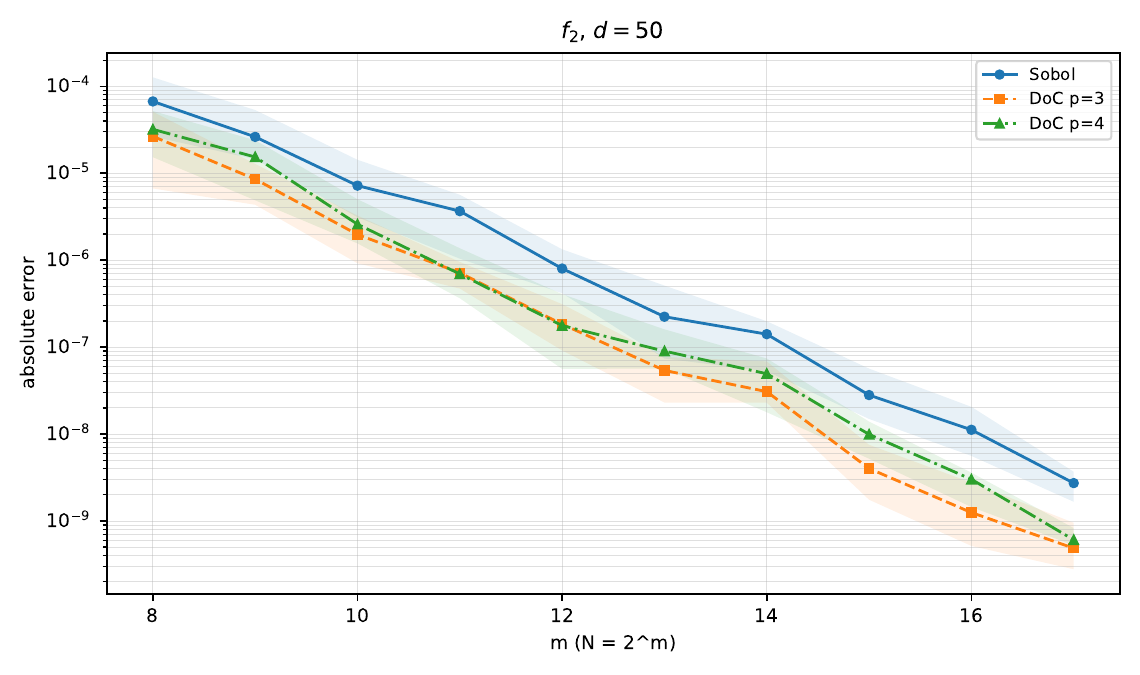}
    \caption{$\alpha=2$.}
  \end{subfigure}
  \par\medskip
  \begin{subfigure}{0.48\linewidth}
    \includegraphics[width=\linewidth]{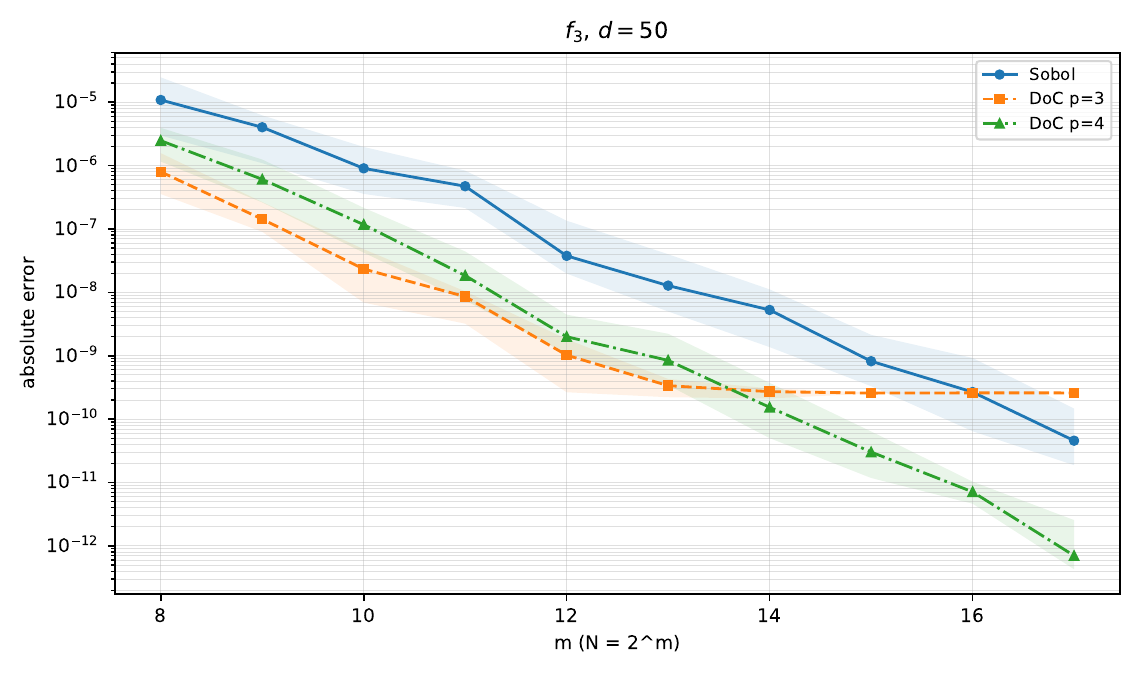}
    \caption{$\alpha=3$.}
  \end{subfigure}
  \caption{Absolute errors for the weighted exponential functions $f_\alpha$
  over $32$ scramblings. Lines show the median of the absolute errors, and
  shaded bands show their 25th--75th percentile range.}
  \label{fig:weighted}
\end{figure}

\FloatBarrier

These experiments are not intended to establish uniform superiority over
scrambled Sobol' rules. Rather, they demonstrate the intended mechanism: when
an integrand is governed by low-dimensional projections that the net recovers
at resolution $p$, replacing the coordinate symbols by cubature nodes can
yield a substantial finite-budget gain. The balance depends on the retained
resolution, the coordinate rule, the underlying net, and the structure of the
integrand; in particular, when the product-cubature bias is large for the
integrand or the chosen resolution, it can offset the gain from projection
recovery.

\section{Discussion and future directions}\label{sec:discussion}

The construction in this paper leaves two coupled design problems fixed: the
retained resolution $p$, and the joint choice of the underlying digital net,
the one-dimensional cubature alphabet, and its labeling. The results suggest
several ways in which these choices could be optimized rather than fixed.

\subsection{Choosing the retained resolution}

The resolution $p$ controls a finite-budget trade-off. Smaller $p$ requires
fewer retained bits for full-grid recovery and can make important projections
recover at smaller budgets. The price is a higher product-cubature floor,
because the coordinate rule has only $2^p$ nodes. Larger $p$ delays recovery,
but gives a lower floor after the relevant projections have been recovered.
This suggests adaptive choices of $p$, or asymptotic constructions with
$p=p(m)$, based on an estimated balance between product-cubature bias and
residual digital-net error.

\subsection{Joint digital-net and cubature design}

The present experiments start from a standard Sobol' net and then replace the
coordinate alphabet. A more systematic construction would choose the digital
structure and the cubature alphabet together. Projection ranks, projected
$t$-values, coordinate-importance weights, scramblings, and node labelings
could all enter such a design criterion.

This direction is also where the relation to cubature designs becomes most
relevant. Coding-theoretic and orthogonal-array cubature constructions place
finite linear patterns on cubature nodes in order to obtain high-dimensional
formulas with polynomial exactness guarantees \cite{Kuperberg2006}. The present
work uses an existing QMC digital net instead, with projection recovery and
residual digital-net error as the main criteria. A natural next step is to seek
arrays that satisfy both digital-net projection criteria and cubature exactness
criteria after the alphabet replacement. Such a joint design should also take
the residual coefficient effect into account. The mapping from retained
$p$-bit labels to cubature nodes should vary regularly, so that it does not
introduce unnecessary oscillation into the transformed integrand.

\appendix
\section{Cubature-node vectors used in the experiments}
\label{app:node-vectors}

For reproducibility, the equal-weight one-dimensional node vectors used in
Section~\ref{sec:experiments} are listed below. The nodes are in increasing
order: label $a\in\{0,\ldots,2^p-1\}$ is assigned to the node $c_a^{(p)}$.
\[
\begin{aligned}
(c_a^{(3)})_{a=0}^{7}
={}&\bigl(
0.05093265814548453,\ 0.20573326961409688,\ 0.29319949903096630,\\
&\phantom{\bigl(}
0.45191153984582166,\ 0.54808846015417834,\ 0.70680050096903370,\\
&\phantom{\bigl(}
0.79426673038590312,\ 0.94906734185451547
\bigr),\\[1mm]
(c_a^{(4)})_{a=0}^{15}
={}&\bigl(
0.02500000000000000,\ 0.10768136111123704,\ 0.14031474924487128,\\
&\phantom{\bigl(}
0.22187628135877777,\ 0.30000000000000000,\ 0.32142317232938330,\\
&\phantom{\bigl(}
0.40993454388816797,\ 0.48000000000000000,\ 0.52000000000000000,\\
&\phantom{\bigl(}
0.59006545611183203,\ 0.67857682767061670,\ 0.70000000000000000,\\
&\phantom{\bigl(}
0.77812371864122223,\ 0.85968525075512872,\ 0.89231863888876296,\\
&\phantom{\bigl(}
0.97500000000000000
\bigr).
\end{aligned}
\]

\end{document}